\documentclass{amsart}
\usepackage{graphicx,amsfonts,amssymb,amsmath,amsthm}
\usepackage[pdftex]{hyperref}
\hypersetup{pdfstartview={XYZ 120 670 1}, pdfpagemode=none}
\usepackage{cite}  % we want this after hyperref
\usepackage{pdfsync}

\theoremstyle{plain} %--default
\newtheorem{theorem}    {Theorem}[section] 
\newtheorem{lemma}      [theorem]{Lemma}
\newtheorem{corollary}  [theorem]{Corollary}

\newtheorem{question}   [theorem]{Question}

\newtheorem*{thm}    {Theorem}
\newtheorem*{thm2}    {Theorem 1.2}

\newtheorem*{thm5}    {Theorem 1.5}

\theoremstyle{definition}

\theoremstyle{remark}
\newtheorem{remark}              {Remark}

\numberwithin{equation}{section}

\def\A{\mathbb A}
\def\C{\mathbb C}

\def\N{\mathbb N}

\def\Q{\mathbb Q}

\begin{document}

\title[On the occurrence of Hecke eigenvalues]{On the Occurrence of Hecke Eigenvalues and a Lacunarity Question of Serre}

\author{Nahid Walji}

\maketitle
\begin{abstract}
Let $\pi$ be a unitary cuspidal automorphic representation for GL(n) over a number field. We establish upper bounds on the number of Hecke eigenvalues of $\pi$ equal to a fixed complex number. For GL(2), we also determine upper bounds on the number of Hecke eigenvalues with absolute value equal to a fixed number $\gamma$; in the case $\gamma = 0$, this answers a question of Serre. These bounds are then improved upon by restricting to non-dihedral representations. Finally, we obtain analogous bounds for a family of cuspidal automorphic representations for GL(3).
\end{abstract}

\section{Introduction}

Let $n$ be a positive integer, $F$ a number field, and let $\mathcal{A}_0({\rm GL}_n(\A_F) )$ be the set of unitary cuspidal automorphic representations for ${\rm GL}_n(\A_F)$.
Denote by $v$ a finite place of $F$ at which $\pi \in \mathcal{A}_0 ({\rm GL}_n(\A_F) )$ is unramified. Associated to such a place $v$ is the Langlands conjugacy class $A_v (\pi)$, which one can represent by a diagonal matrix of Satake parameters $\{\alpha _{1,v}, \dots , \alpha _{n,v}\}$. We will denote the trace of such a matrix as $a_v (\pi)$.

We begin with the following question: if we fix a non-negative number $\gamma$, what can be said about the density of the set $\{v \mid |a_v(\pi)| \neq \gamma\}$?

It is instructive to know what is predicted. For $\gamma = 0$, Serre raised the
following question:

\begin{question}[Serre] \label{conja}
Let $\pi \in \mathcal{A}_0({\rm GL}_n(\A_F))$ and consider the set of places $v$ at which the Hecke eigenvalue is zero. Is the density of this set, if it exists, bounded above by $1-1/n^2$?
\end{question}

This originated in \textit{Probl\`eme} on page 372, Section 6 of~\cite{Se81}. In that section, Serre considers the lacunarity of a Dirichlet series associated to an $r$-dimensional irreducible $\ell$-adic Galois representation over a number field. He proves that there exists an upper bound of $1- 1/r^2$ for the density of places at which the associated trace of Frobenius is zero, and shows that such a bound is sharp. He ends the section by asking whether the same bounds might hold on the automorphic side, adding that he expects these to follow from the conjectures of Langlands associating a reductive group $H(\pi)$ to a given automorphic representation $\pi$.

We will address this Question and show that the bound holds in various cases.
We actually establish an upper bound on the upper Dirichlet density, so it will not matter if the set does not have a density.

We will show that, for each $n$, the Ramanujan conjecture for GL(n) implies a positive answer to Question~\ref{conja} for GL(n). We also explain why the bound in Question~\ref{conja}, if true, would be sharp. This is due to an application of a theorem of Arthur--Clozel to determine the automorphy of the examples of Galois representations from~\cite{Se81}.

We obtain unconditional results for Question~\ref{conja} when $n = 2$, as well as for a family of automorphic representations when $n = 3$. 
Note that a cuspidal automorphic representation $\pi$ can be twisted by some $|\cdot|^t \circ {\rm det}$, for a suitable real number $t$, in order to be made unitary. So we can restrict to considering unitary cuspidal automorphic representations.
Recall that an automorphic representation $\Pi$ for ${\rm GL}_n(\A_F)$ is called \textit{essentially self-dual} if there exists a Hecke character $\chi$ such that $\Pi \simeq \widetilde{\Pi} \otimes \chi$.
\begin{theorem}\label{thmb}
Let $\pi \in \mathcal{A}_0({\rm GL}_n(\A_F))$, where $n = 2$ or where $\pi$ is essentially self-dual and $n = 3$.
Then the upper density of places $v$ at which the Hecke eigenvalue is zero is bounded above by $1-1/n^2$, where $n = 2$ or $3$, respectively.
\end{theorem}

For $\pi \in \mathcal{A}_0({\rm GL}_n(\A_F))$, let $T = T (\pi)$ be a finite set of places containing the places at which $\pi$ is ramified as well as the infinite places. When we have an $L$-function $L(s)$ which can be expressed in some right-half plane as an Euler product over the places of $F$, we let $L^T(s)$ denote, where convergent, the Euler product over places $v \not \in T$. Given a non-negative number $\gamma$, we define $S_\gamma := \{v \mid |a_v(\pi)|\neq \gamma\}$, and we let $\underline{\delta}(S_\gamma)$ denote the lower Dirichlet density of such a set.

In the following two theorems, the conditions mentioned will include bounds towards the Ramanujan conjecture. 
Let us define $\theta = \theta (\pi)$ to be a constant such that $|\alpha_{i,v}(\pi)| \leq Nv^\theta$ for all $i$ and for $v \not \in T$. 
Known bounds for $\pi \in \mathcal{A}_0({\rm GL}_n(\A_F))$  are $\theta = 1/2 - 1/(n ^2 + 1)$ for $n \geq 5$, and $\theta = 7/64, 5/14,$ and $9/22$ for $n = 2,3,$ and $4$, respectively~\cite{LRS99, Ki03, BB11}.

Theorem~\ref{thmb} will follow from:
\begin{theorem} \label{thmc}
Given $\pi \in \mathcal{A}_0({\rm GL}_n(\A_F))$ such that $\theta (\pi) < 1/4$, assume that $L^T(s, \pi \times \pi \times \overline{\pi} \times \overline{\pi})$ has absolutely convergent Euler product for $s > 1$ and pole of order $m$ at $s=1$. Then 
\begin{align*}
\underline{\delta} (S_\gamma ) \geq 1- \frac{m-1 }{m - 2 \gamma^2 + \gamma^4}.
\end{align*}
\end{theorem}

\begin{remark}
For $n \geq 2$, we expect that $m \geq 2$. Given $\pi \in \mathcal{A}_0({\rm GL}_n(\A_F))$, one expects the existence of an automorphic tensor product $\pi \boxtimes \overline{\pi}$ with an isobaric decomposition of ${\rm Ad}\pi \boxplus 1$. If ${\rm Ad}\pi$ is cuspidal then $m = 2$, whereas if it decomposes further then $m$ may take larger values. 
\end{remark}

When $n = 2$ we will show that $m \leq 4$, which gives bounds of 1/4 for $\gamma = 0$ and 3/4 for $\gamma = 2$, both of which are sharp. This will be demonstrated in Section~\ref{sce} by constructing a dihedral automorphic representation that is associated to an Artin representation whose image is the quaternion group. If $\pi \in \mathcal{A}_0(GL_2(\A_F))$ is non-dihedral, then ${\rm Ad}\pi$ is cuspidal and so $m = 2$.

We also obtain bounds for essentially self-dual representations for GL(3). 
First, let us note that any $\pi \in \mathcal{A}_0(GL_3(\A_F))$ that is essentially self-dual is isomorphic to a twisted adjoint lift of some $\tau \in \mathcal{A}_0(GL_2(\A_F))$. In this case, we will say that $\pi$ is \textit{associated} to $\tau$. 
We also point out that a cuspidal automorphic representation for GL(2) is said to be \textit{solvable polyhedral} if its $L$-function is equal to that of an Artin representation that is either dihedral, tetrahedral, or octahedral (see~\cite{RW04}).

Now when $n = 3$ and $\pi$ is essentially self-dual and associated to a cuspidal representation for GL(2) that is not of solvable polyhedral type, we will show that $m = 3$ and so the bound will be 1/3 for $\gamma = 0$. This is stronger than the conjectured bound of 1/9 for GL(3), because we have excluded those representations which are solvable polyhedral or non-essentially self-dual.

For $n \geq 4$, we are not able to establish the meromorphic continuation and 
order of a pole at $s=1$ for  $L^T(s, \pi \times \pi \times \overline{\pi} \times \overline{\pi})$.
Crucially, one does not know (and we seem to be far from such a result) about the automorphy of the adjoint lift (of degree $n ^2 -1$) for general $n$. 

Theorem~\ref{thmb} will follow from Theorem~\ref{thmc},
and the fact that the Ramanujan conjecture implies a positive answer to Question~\ref{conja}. This is explained in Section~\ref{sthmb}.\\

The techniques in the proof of Theorem~\ref{thmc} can be extended to make use of higher degree product $L$-functions so as to obtain new bounds.

\begin{theorem}\label{thmd}
Given $\pi \in \mathcal{A}_0({\rm GL}_n(\A_F))$ such that $\theta (\pi) < 1/8$, assume that $L^T(s, \pi \times \pi \times \overline{\pi} \times \overline{\pi}), L^T(s, \pi ^{\times 3}\times \overline{\pi}^{\times 3})$ and $L^T(s, \pi ^{\times 4}\times \overline{\pi}^{\times 4})$ have absolutely convergent Euler products for $s > 1$ and poles of order $2, m'$ and $m$, respectively, at $s=1$. 
Then 
\begin{align*}
\underline{\delta}(S_\gamma) \geq 1  - \frac{m - m'^2 +4m' -8}{\gamma^8 + (4-2m')\gamma^6 + (2m' + m -12)\gamma^4 + (4m'-2m)\gamma^2 + (2m-m'^2)}.
\end{align*}
\end{theorem}

\begin{remark}
This version of Theorem~\ref{thmd} in the case of {\rm GL}(2) was indicated by a referee. Our assumption that $L^T(s, \pi \times \pi \times \overline{\pi} \times \overline{\pi})$ has a pole of order 2 at $s=1$ (which would be implied by the existence of a cuspidal adjoint lift) was made for ease of exposition; the general case would arise in a similar manner.
\end{remark}

The work of Kim--Shahidi~\cite{KS00,KS02} and Kim~\cite{Ki03} on the automorphy and cuspidality of the symmetric third and fourth power lifts will imply that for a unitary cuspidal automorphic representation $\pi$ for GL(2) that is not of solvable polyhedral type, we have $m' = 5$ and $m = 14$:

\begin{corollary}
Let $\pi$ be a unitary cuspidal automorphic representation for GL(2) that is not of solvable polyhedral type, then
\begin{align*}
\underline{\delta}(S_\gamma) \geq 1 - \frac{1}{3 + \gamma^2 (\gamma^2-2)^3}.
\end{align*}
\end{corollary}

\begin{remark}
For the family of unitary cuspidal automorphic representations for GL(2) that are not of solvable polyhedral type, Theorem~\ref{thmd} improves on the bounds from Theorem~\ref{thmc} for all but four values of $\gamma$ (those being $\pm (1/2)\sqrt{3 \pm \sqrt{5}}$). In particular, in the case of $\gamma = 0$ the bound improves from 1/2 to 2/3.
\end{remark}

It is also possible to obtain bounds that hold unconditionally for all unitary cuspidal automorphic representations for GL(n). If we focus on the occurrence of Hecke eigenvalues equal to a complex number $\alpha$, rather than the occurrence of Hecke eigenvalues with a given absolute value, then we obtain:

\begin{theorem}\label{propf}
Given $\pi \in \mathcal{A}_0({\rm GL}_n(\A_F))$, fix a complex number $\alpha$ and let $S=S(\pi, \alpha)=\{v \mid a_v(\pi) \neq \alpha \}$. Then  
 \begin{align*}
   \underline{\delta}(S) \geq \frac{|\alpha|^2}{|\alpha|^2+ 1}.
 \end{align*}
 \end{theorem}

\begin{remark}
This inequality is sharp for $n = 1$ when $\alpha = + 1$ or $-1$. For example, let $\pi$ be any Hecke character of order two.
Since the inequality holds for all GL(n), one cannot expect a non-trivial bound for $\alpha = 0$ (see Remark~\ref{rka} of Section~\ref{pt6}).
\end{remark}

Returning to the set of examples that demonstrate that the bound in Question~\ref{conja} is optimal, we also use them to prove that (a suitable reinterpretation of) a theorem of Ramakrishnan~\cite{Ra97} is sharp. This theorem concerns the occurrence of primes at which the weak Ramanujan conjecture holds. Being sharp, any improvement in the bounds must then apply to a family of cuspidal automorphic representations that excludes the examples mentioned.

Fix $\pi \in \mathcal{A}_0({\rm GL}_n(\A_F))$. By~\cite{Ra97}, the set of places $v$ where $|a_v(\pi)|\leq k$ has a lower Dirichlet density of at least $(k^2 -1)/k^2$. In particular, setting $k = n$ gives strong lower bounds for the lower Dirichlet density of the set of places for which the Hecke eigenvalues satisfy the weak Ramanujan conjecture.
 However, the bounds weaken as $k$ tends to 1, and when $k \in [0,1]$, the inequality does not provide any information. 

The objective of Theorem~\ref{propf} is to shed light on the occurrence of Hecke eigenvalues which have absolute value less than or equal to 1. The trade-off is that it provides results about the occurrence of a fixed complex number, rather than elements within an interval.

Our proofs make use of key results in~\cite{KS02},~\cite{Ki03}, and~\cite{Ra00}, as well as techniques from~\cite{Ra94}, ~\cite{Ra97}, and~\cite{NW2}.\\

This paper is structured as follows.
In Section~\ref{s2}, we prove Theorems~\ref{thmc} and~\ref{thmd}. In Section~\ref{s3}, we show that certain families of cuspidal automorphic representations for GL(2) and GL(3) satisfy the conditions of either Theorem~\ref{thmc} or Theorem~\ref{thmd}. In Section~\ref{s4}, we prove Theorem~\ref{propf}. In Section~\ref{sce}, we address Question~\ref{conja}, explaining how it is implied by the Ramanujan conjecture and why it would be sharp. In Section~\ref{sthmb}, we complete the proof of Theorem~\ref{thmb}.\\

\section{Proof of Theorems~\ref{thmc} and~\ref{thmd}}\label{s2}

Given a number field $F$ and a set of places $S$, the \emph{lower Dirichlet density} of $S$ is 
\begin{align*}
\underline{\delta}(S) = \lim_{s \rightarrow 1^+} {\rm inf} \frac{\sum_{v \in S}{\rm N}v^{-s}}{\log\left( 1/(s-1) \right)}.
\end{align*}
Note that if $S$ has a Dirichlet density, then it coincides with its lower Dirichlet density.\\

Given an $L$-function with an Euler product over the set of places of $F$
\begin{align*}
  L(s) = \prod_{v}L_v(s),
\end{align*}
for a finite set of places $T$, we define the incomplete $L$-function 
\begin{align*}
  L^T(s) = \prod_{v \not \in T}L_v(s).
\end{align*}

Fix some $\pi \in \mathcal{A}_0({\rm GL}_n(\A_F))$ that satisfies the conditions of Theorem~\ref{thmc}, so that $L^T(s, \pi \times \pi \times \overline{\pi} \times \overline{\pi})$ has meromorphic continuation to $s = 1$ with a pole of order $m$ there and $\pi$ satisfies suitable bounds towards the Ramanujan conjecture.
Let $T$ be the set of infinite places as well as those finite places at which $\pi$ is ramified.\\

The condition that $\theta (\pi) < 1/4$ in Theorem~\ref{thmc} implies that
\begin{align}\label{expr1}
\log L^T(s, \pi \times \overline{\pi})
= \sum_{v \not \in T}\frac{|a_v(\pi)|^2}{Nv^s} + O\left(1\right) 
\end{align}
as $s \rightarrow 1^+$, and (using positivity)
\begin{align}\label{expr2a}
\log L^T(s, \pi \times \pi \times \overline{\pi} \times \overline{\pi})
\geq \sum_{v \not \in T}\frac{|a_v(\pi)|^4}{Nv^s}
\end{align}
for $s > 1$. 

Fix a non-negative number $\gamma$ and a set $S=S_\gamma = \{v \mid |a_v(\pi)|\neq \gamma\}$. Denote by $1_S(v)$ the indicator function on the set of finite places with respect to the set $S$.
For any fixed $s > 1$, the sequences $\left( \frac{ |a_v(\pi)|^{2} - \gamma^{2} }{ Nv^s}\right)$ and $\left(\frac{ 1_S (v) }{Nv^s}\right)$ are elements of the space $\ell^2$. We apply Cauchy--Schwarz for $\ell^2$ to obtain
\begin{align} \label{cs}
\left| \sum_{v \not \in T}\frac{(|a_v(\pi)|^{2} - \gamma^{2}) 1_S(v)}{Nv^s} \right| \leq \left(\sum_{v \not \in T}\frac{(|a_v(\pi)|^{2} - \gamma^{2})^2}{Nv^s}\right)^{1/2} \left(\sum_{v \in S}\frac{1}{Nv^s}\right)^{1/2}.
\end{align}
One knows that $L^T(s, \pi \times \overline{\pi})$ has a simple pole at $s=1$~\cite{JS81}. This, along with the conditions on the $L$-functions in Theorem~\ref{thmc},   
implies that
\begin{align*}
\log L^T(s, \pi \times \overline{\pi})
&= \log \left(\frac{1}{s-1}\right) + O\left(1\right),
\end{align*}
and
\begin{align*}
\log L^T(s, \pi \times \pi \times \overline{\pi} \times \overline{\pi})
&= m \log \left(\frac{1}{s-1}\right) + O\left(1\right).
\end{align*}

We put all this together to determine the limit inferior as $s \rightarrow 1^+$ of equation~(\ref{cs}) where both sides have been divided by $\log \left(1/ (s-1)\right)$.
The resulting inequality is
\begin{align*}
 \frac{(1 - \gamma^2)^2 }{m - 2 \gamma^2 + \gamma^4} \leq  \underline{\delta} (S)
\end{align*}
which completes the proof of Theorem~\ref{thmc}. \\

The condition in Theorem~\ref{thmd} that $\theta (\pi) < 1/8$ implies that 
\begin{align}\label{expr2b}
\log L^T(s, \pi \times \pi \times \overline{\pi} \times \overline{\pi})
&= \sum_{v \not \in T}\frac{|a_v(\pi)|^4}{Nv^s} + O\left(1\right),\\
\log L^T(s, \pi ^{\times 3} \times \overline{\pi} ^{\times 3})
&= \sum_{v \not \in T}\frac{|a_v(\pi)|^6}{Nv^s} + O\left(1\right),
\end{align}
as $s \rightarrow 1^+$, and
\begin{align}\label{expr3}
\log L^T(s, \pi ^{\times 4} \times \overline{\pi} ^{\times 4})
\geq \sum_{v \not \in T}\frac{|a_v(\pi)|^8}{Nv^s}.
\end{align}
for $s > 1$.

The proof of Theorem~\ref{thmd} proceeds in a similar manner to that of Theorem~\ref{thmc}, except now we use Cauchy-Schwarz to construct
\begin{align*}
\left| \sum_{v \not \in T}\frac{(|a_v(\pi)|^{2} - \gamma^{2}) (|a_v(\pi)|^2 -c) 1_S(v)}{Nv^s} \right| \leq \left(\sum_{v \not \in T}\frac{(|a_v(\pi)|^{2} - \gamma^{2})^2(|a_v(\pi)|^2 -c)^2}{Nv^s}\right)^{1/2} \\
\cdot \left(\sum_{v \in S}\frac{1}{Nv^s}\right)^{1/2}, 
\end{align*}
and we find (through elementary analysis) that the optimal choice for $c$ is 
\begin{align*}
c = \frac{\gamma^6 + (4-2m')\gamma^4 + (m+m' -8)\gamma^2 (-m + 2m')}{-\gamma^4 + (m'-2)\gamma^2 + (4-m')},
\end{align*}
leading to the bound 
\begin{align*}
\underline{\delta}(S_\gamma) \geq \frac{\gamma^8 + (4-2m')\gamma^6 + (2m' + m -12)\gamma^4 + (4m'-2m)\gamma^2 + (m-4m' + 8)}{\gamma^8 + (4-2m')\gamma^6 + (2m' + m -12)\gamma^4 + (4m'-2m)\gamma^2 + (2m-m'^2)}.\\
\end{align*}

\section{$L$-functions for GL(2) and GL(3)}\label{s3}

In this section we prove two things. First, we show that the $L$-function conditions in the statement of Theorem~\ref{thmc} are satisfied by specified families of unitary cuspidal automorphic representations for GL(2) and GL(3). Second, we show that the $L$-function conditions in Theorem~\ref{thmd} are satisfied by a certain family of unitary cuspidal automorphic representations for GL(2).

\subsection{Cuspidal representations for GL(2)}\ \\

We begin by remarking that the bounds of~\cite{KS02} or~\cite{BB11} towards the Ramanujan conjecture for GL(2) satisfy that condition for both Theorem~\ref{thmc} and Theorem~\ref{thmd}.\\

\begin{lemma} \label{lema}
For $\pi \in \mathcal{A}_0(GL_2(\A_F))$, let $T$ be the set that contains exactly all the infinite places as well as those finite places where $\pi$ is ramified. 
At $s=1$, $L^T(s, \pi \times \overline{\pi})$ has a simple pole and $L^T(s, \pi \times \pi \times \overline{\pi} \times \overline{\pi})$ has a pole of order at most four.
\end{lemma}

\begin{proof}
For any Rankin--Selberg $L$-function for GL(m) $\times$ GL(n), the Euler factor associated to a place $v$ is invertible in ${\rm Re}(s)\geq 1$.
Removing a finite number of factors from the Euler product does not affect the presence or order of a pole at $s=1$.
Therefore, the simple pole of $L^T(s, \pi \times \overline{\pi})$ follows from Rankin--Selberg theory~\cite{JS81}.\\

The analysis of $L^T(s, \pi \times \pi \times \overline{\pi} \times \overline{\pi})$ depends on whether $\pi$ is dihedral.
First, assume that $\pi$ is non-dihedral. Since $\pi \boxtimes \overline{\pi} \simeq {\rm Ad}\pi \boxplus 1$ (the automorphy on the left-hand side is due to~\cite{Ra00}),
\begin{align*}
 L^T(s, (\pi \boxtimes \overline{\pi}) \times (\pi \boxtimes \overline{\pi})) = L^T(s, {\rm Ad}\pi \times {\rm Ad}\pi)L^T(s, {\rm Ad}\pi)^2 \zeta^T_F(s). 
\end{align*}
 The adjoint representation is cuspidal if and only if $\pi$ is non-dihedral~\cite{GJ78}, and since it is self-dual, the $L$-function on the left-hand side has a pole of order two at $s=1$.\\

Let $\pi$ now be dihedral. We introduce some notation for this case.
A dihedral automorphic representation $\pi$ for $GL_2(\A_F)$ can be expressed as the automorphic induction from $E$ to $F$ of the Hecke character $\mu$, where $E$ is a quadratic extension of $F$. Denote this as $\pi = I^F_E(\mu)$. Note that in certain cases, the same dihedral representation may be induced from more than one quadratic extension.

We will make use of a theorem of Ramakrishnan (from \cite{Ra00}; see also \cite{PR12}) to analyse the incomplete $L$-function.

\begin{thm}
Given $\pi, \pi' \in \mathcal{A}_0(GL_2(\A_F))$ that are both dihedral, the automorphic tensor product $\pi \boxtimes \pi'$ is a cuspidal automorphic representation for $GL_4(\A_F)$ if and only if $\pi$ and $\pi'$ cannot be induced from the same quadratic extension.

On the other hand, if $\pi$ and $\pi'$ can be induced from the same quadratic extension $E$, express them as $\pi = I_E^F (\mu)$ and $\pi' = I_E^F (\nu)$, for suitable Hecke characters $\mu$ and $\nu$. Then,
\begin{align*}
  \pi \boxtimes \pi' \simeq I_E^F(\mu \nu) \boxplus I_E^F(\mu \nu^\tau),
\end{align*}
where $\tau$ denotes the non-trivial element of ${\rm Gal}(E/F)$.
\end{thm}

When $\pi' \simeq \overline{\pi}$, we have
\begin{align*}
  \pi \boxtimes \overline{\pi} \simeq 1 \boxplus \chi \boxplus  I_E^F(\nu / \nu^\tau),
\end{align*}
where $\chi$ is the Hecke character associated to the extension $E/F$.
If $\nu / \nu^\tau$ is not ${\rm Gal}(E/F)$-invariant, then $ I_E^F(\nu / \nu^\tau)$ is cuspidal and $L^T(s, (\pi \boxtimes \overline{\pi}) \times (\pi \boxtimes \overline{\pi}))$ has a pole of order at most three. On the other hand, if $\nu / \nu^\tau$ is ${\rm Gal}(E/F)$-invariant, then 
\begin{align*}
  \pi \boxtimes \overline{\pi} \simeq 1 \boxplus \chi \boxplus \nu / \nu^\tau \boxplus \left(\nu / \nu^\tau \cdot \chi \right).
\end{align*}
So $L^T(s, (\pi \boxtimes \overline{\pi}) \times (\pi \boxtimes \overline{\pi}))$ has a pole at $s=1$ of order at most four.
\end{proof}

Now let us turn to two of the $L$-functions mentioned in Theorem~\ref{thmd}.

\begin{lemma}
Given $\pi \in \mathcal{A}_0(GL_2(\A_F))$ that is not of solvable polyhedral type, the incomplete $L$-functions $L^T(s, \pi ^{\times 4} \times \overline{\pi} ^{\times 4})$ and  $L^T(s, \pi ^{\times 3} \times \overline{\pi} ^{\times 3})$ have poles at $s=1$ of orders 14 and 5, respectively.
\end{lemma}

\begin{proof}
By the Clebsch-Gordon decomposition of tensor products, we know that 
\begin{align*}
L^T(s, \pi ^{\times 4} \times \overline{\pi}^{\times 4}) =
L^T(s, A ^4 \pi \times A ^4 \pi) L^T(s, {\rm Ad}\pi \times {\rm Ad}\pi)^9 \zeta_F^T(s)^4
\end{align*}
where $A^4 \pi = {\rm Sym}^4 \pi \otimes \overline{\omega}^2$, with $\omega$ being the central character of $\pi$.

From Kim and Kim--Shahidi~\cite{Ki03,KS02}, we know that $A ^4 \pi$ is automorphic and in fact it is cuspidal, since $\pi$ is not of solvable polyhedral type. 
By Gelbart--Jacquet~\cite{GJ78}, ${\rm Ad}\pi$ is automorphic, and because $\pi$ is not of dihedral type we have that ${\rm Ad}\pi$ is cuspidal.
Since they are both self-dual, the $L$-function $L^T(s, \pi ^{\times 4} \times \overline{\pi}^{\times 4})$ has a pole of order 14 at $s=1$.

The proof for $L^T(s, \pi ^{\times 3} \times \overline{\pi} ^{\times 3})$ follows similarly.
\end{proof}

\subsection{Essentially self-dual cuspidal representations for GL(3)}\label{gl3}
In this subsection, we show the existence of a family of cuspidal automorphic representations for GL(3) that satisfy the conditions of Theorem~\ref{thmc} (with $m \leq 3$).

We begin with the following theorem, known to the experts:

\begin{thm}
Given a self-dual cuspidal automorphic representation $\Pi$ for GL(3), there exists a cuspidal automorphic representation $\pi$ for GL(2) and a Hecke character $\nu$ of order (at most) two such that 
\begin{align*}
\Pi \simeq {\rm Ad}\pi \otimes \nu.
\end{align*}
Furthermore, $\pi$ is unique up to character twist.
\end{thm}
Note that $\nu$ is the central character of $\Pi$ and that $\pi$ must necessarily be non-dihedral (otherwise the adjoint lift would not be cuspidal).

\begin{remark}
This theorem was known in the folklore to arise from a comparison of the stable trace formula for $SL(2)/F$ with the twisted trace formula for $PGL(3)/F$. More recently, another proof was found by Ramakrishnan~\cite{Ra09}, which makes use of a descent of Ginzburg, Rallis, and Soudry~\cite{GRS99} and a forward transfer from odd orthogonal groups to GL(n) of Cogdell, Kim, Piatetski-Shapiro, and Shahidi~\cite{CKPSS04}.
\end{remark}

An immediate consequence of the theorem above is that any essentially self-dual cuspidal automorphic representation $\Pi$ for GL(3) can be expressed as ${\rm Ad}\pi \otimes \eta$, where $\pi$ is a non-dihedral automorphic representation for GL(2) and $\eta$ is a Hecke character. 
Given two such cuspidal automorphic representations $\Pi$ and $\pi$, we will say that $\Pi$ is \textit{associated} to $\pi$.

We now require the following lemma:

\begin{lemma}
Let $\Pi$ be an essentially self-dual unitary cuspidal automorphic representation for GL(3) associated to some $\pi$.
Denote the central character of $\pi$ by $\omega$, and let $T$ be the set containing exactly all the infinite places as well as the finite places where $\pi$ is ramified. Then
\begin{align*}
L^T(s, \Pi \times \Pi \times \overline{\Pi}\times \overline{\Pi}) &= L^T(s, (A^4 \pi \boxplus {\rm Ad}\pi \boxplus 1)\times (A^4 \pi \boxplus {\rm Ad}\pi \boxplus 1))
\end{align*}
\end{lemma}
where $A^4 \pi \simeq {\rm Sym}^4 \pi \otimes \omega ^{-2}$. 
\begin{proof}
This follows by examining the decompositions of tensor powers of Langlands conjugacy classes.
Given $v \not \in T$, let us represent $A_v(\pi)$ by the matrix ${\rm diag}\{\alpha,\beta\}$. Writing $\Pi$ as ${\rm Ad}\pi \otimes \eta$ (as explained above), we can represent $A_v(\Pi)$ by ${\rm diag}\{\alpha \eta / \beta, \eta , \beta \eta / \alpha\}$. For $A_v(\Pi \times \overline{\Pi})$, note that 
\begin{align*}
&       \left( \begin{array}{ccc}
\alpha \eta /\beta          &  &  \\
         & \eta  &  \\
         &  & \beta \eta / \alpha
       \end{array} \right)
\otimes 
     \left( \begin{array}{ccc}
     \beta \eta / \alpha  &  &  \\
       & \eta &  \\
       &  & \alpha \eta / \beta
     \end{array} \right)
\end{align*}
 is equivalent to
\begin{align*}
&     \left( \begin{array}{ccccc}
     \alpha^2 / \beta^2  &  &  & & \\
       & \alpha / \beta &  & & \\
       &  & 1 &  &\\
       &  &  & \beta / \alpha &\\
       &  &  &  &  \beta^2 / \alpha^2
     \end{array} \right)
\oplus
     \left( \begin{array}{ccc}
     \alpha/ \beta  &  &  \\
       & 1 &  \\
       &  & \beta/\alpha
     \end{array} \right)
\oplus 
1.
\end{align*}
Therefore,
\begin{align*}
  L^T(s, \Pi \times \overline{\Pi}) = L^T(s, A^4 \pi \boxplus {\rm Ad}\pi \boxplus 1)
\end{align*}
and the lemma follows.
\end{proof}

Recall that $\pi$ is a cuspidal automorphic representation that is not of solvable polyhedral type. Thus, by Gelbart--Jacquet~\cite{GJ78}, the adjoint lift ${\rm Ad}\pi$ is cuspidal. By Kim--Shahidi~\cite{KS02}, the symmetric fourth power lift of $\pi$ is a cuspidal automorphic representation for GL(5), and so $A^4 \pi$ is cuspidal. Furthermore, $A^4 \pi$ is self-dual.
The identity in the lemma above therefore implies that $L^T(s, \Pi \times \Pi \times \overline{\Pi}\times \overline{\Pi})$ has a pole of order three.\\

As in the previous subsection, we note that the known bounds towards the Ramanujan conjecture for GL(2) (see~\cite{KS02},~\cite{BB11}) satisfy the conditions of Theorem~\ref{thmc} for essentially self-dual unitary cuspidal automorphic representations for GL(3).

\section{Proof of Theorem~\ref{propf}}\label{pt6}\label{s4}
We prove the final theorem from the introduction:
\begin{thm5}
Given $\pi \in \mathcal{A}_0({\rm GL}_n(\A_F))$, fix a complex number $\alpha$ and let $S=S(\pi, \alpha) = \{v \mid a_v(\pi) \neq \alpha \}$.
 Then 
 \begin{align*}
   \underline{\delta}(S) \geq \frac{|\alpha|^2}{|\alpha|^2 + 1}.
 \end{align*}
\end{thm5}

\begin{proof}
We begin by establishing the asymptotic properties of two particular Dirichlet series as $s \rightarrow 1^+$.
Let $T$ be the set containing the infinite places of $F$, as well as the finite places where $\pi$ is ramified. 
The bounds towards the Ramanujan conjecture established by Luo--Rudnick--Sarnak~\cite{LRS99}
imply 
\begin{align} \label{logeqn}
 \sum_{v \not \in T}\frac{a_v(\pi)}{Nv^s} =  \log L(s, \pi) + O\left( 1\right)
\end{align}
as $s \rightarrow 1^+$. The $L$-function $L(s, \pi)$ is invertible at $s=1$~\cite{JS76}, and so the right-hand side of~(\ref{logeqn}) is bounded as $s \rightarrow 1^+$. The same holds for the Dirichlet series on the left-hand side. 

To address the asymptotic behaviour of 
\begin{align*}
  \sum \frac{|a_v(\pi)|^2}{Nv^s},
\end{align*}
 the bound of Luo--Rudnick--Sarnak will not be sufficient (one would need a bound with exponent less than 1/4 rather than only less than 1/2). Instead, we can establish an upper bound on the rate of growth under this limit, as in Lemma 1.5 of~\cite{Ra97}, using the positivity of the series to show that  
 \begin{align*}
  \sum_{v \not \in T}\frac{|a_v(\pi)|^2}{Nv^s} \leq \log L^T(s, \pi \times \overline{\pi}) = \log \left(\frac{1}{s-1}\right) + O\left(1\right)
 \end{align*}
as $s \rightarrow 1^+$.

Let $1_{S} (v)$ be the indicator function for the set $S=S(\pi, \alpha)$. Then, applying Cauchy--Schwarz,
\begin{align*}
\left|  \sum_{v} \frac{ (a_v(\pi) - \alpha) 1_S (v)}{ Nv^s}\right| \leq \left( \sum_{v} \frac{|a_v(\pi) - \alpha|^2 }{ Nv^s}\right)^{1/2} \left(\sum_{v \in S}\frac{1}{Nv^s}\right)^{1/2}.
\end{align*}

We divide the inequality by $\log (1/ (s-1))$ and examine the limit infimum as $s \rightarrow 1^+$, applying our results above. We obtain
\begin{align*}
    \frac{|\alpha|^2}{|\alpha|^2+ 1} \leq \underline{\delta}(S).
\end{align*}
  
\end{proof}

\begin{remark}\label{rka}
As mentioned in the introduction, this inequality does not provide any information when $\alpha = 0$. The arguments in the proof, and thus the resulting inequality, apply for all GL(n). However, the examples in Section~\ref{sce} show that there exists a sequence $(\pi_n)_{n \in \N}$, where $\pi_n \in \mathcal{A}_0({\rm GL}_n(\A_F))$, such that $\underline{\delta}(S(\pi_n,0)) \rightarrow 0$ as $n \rightarrow \infty$. Therefore, a non-trivial bound for $\alpha = 0$ is not possible for an equation that holds in this much generality.
Obtaining a non-trivial bound would require finding a way to establish a bound that only holds for a finite number of integers $n$. 
\end{remark}

\section{Question~\ref{conja} and examples}\label{sce}\label{s5}
 
Let $E/K$ be a Galois extension of number fields where the Galois group is an $\ell$-adic Lie group. Fix a positive integer $r$ and a finite extension $F$ of $\Q_\ell$. Let $\rho : {\rm Gal}(\overline{K}/ K)\rightarrow {\rm GL}_r(F)$ be a homomorphism that factors through ${\rm Gal}(E/K)$, and denote by $\lambda$ the density of $\{ v \mid {\rm trace } (\rho ({\rm Frob}_v)) = 0\}$. Proposition 16 on page 371 of~\cite{Se81} states that 
\begin{align*}
 \lambda \leq 1- \frac{1}{r^2}.
\end{align*}

Serre explains that this bound is sharp for every value of $r$. The bound is attained by $\rho$ if and only if two particular conditions are satisfied: the representation must be absolutely irreducible, and the projective image in ${\rm PGL}(r)$ must be a group of order $r ^2$. For example, the product of an abelian group of order $r$ with itself will suffice. Furthermore, such examples exist for every $r$.

The section ends with the question of whether analogous results hold for cuspidal automorphic representations associated to reductive groups, even for those which do not correspond to $\ell$-adic representations. He states that it is likely and that such a question is related to the conjectures of Langlands.\\

We consider the case when the reductive group in question is GL(n). Fixing a positive integer $k$, we show how a positive answer to Question~\ref{conja} for GL(k) follows from the Ramanujan conjecture for GL(k).
For a unitary cuspidal automorphic representation $\pi$ for GL(n), the Rankin--Selberg $L$-function $L(s, \pi \times \widetilde{\pi})$ is meromorphic in $\C$, non-vanishing in ${\rm Re}(s) > 1$, and has a simple pole at $s=1$. 
Let $T$ be the set that contains exactly the finite places at which $\pi$ is ramified as well as all the infinite places. 
The Ramanujan conjecture for GL(n) implies that, as $s \rightarrow 1^+$,
\begin{align*}
 \sum_{v \not \in T}\frac{|a_v(\pi)|^2}{Nv^s} = \log L^T(s, \pi \times \overline{\pi}) + O\left(1\right)
\end{align*}
and thus 
\begin{align*}
 \sum_{v \not \in T}\frac{|a_v(\pi)|^2}{Nv^s} = \log \left(\frac{1}{s-1}\right) + O\left(1\right).
\end{align*}
Let $S$ be the set of places (outside $T$) where $a_v(\pi)$ is non-zero.
The (weak) Ramanujan conjecture implies 
\begin{align*}
\sum_{v \not \in T} \frac{|a_v(\pi)|^2}{Nv^s} \leq \sum_{v \in S} \frac{n ^2}{Nv^s}.
\end{align*}
Dividing both sides by $\log \left(1/(s-1)\right)$ and taking the limit infimum as $s \rightarrow 1^+$, we obtain 
\begin{align*}
\frac{1}{n ^2} \leq \underline{\delta}(S)
\end{align*}
which corresponds to the statement in Question~\ref{conja}.\\

We now comment on whether this conjectured bound would be sharp.
As mentioned above, there exists an absolutely irreducible representation $\tau_r : H \rightarrow {\rm GL}_r(\C)$ with the following two properties: first, the projective image is the product of an abelian group of order $r$ with itself, and second, the density of the set of places where the associated trace of Frobenius is zero is exactly $1- 1/r^2$. The group $H$ then factors through a nilpotent group, so we appeal to the following theorem of Arthur--Clozel~\cite{AC89}. 
\begin{thm}
 Given an irreducible Artin representation $\rho$ of degree $n$ that factors through a nilpotent group, there exists a cuspidal automorphic representation $\pi$ for GL(n) such that 
 \begin{align*}
   L(s, \rho) = L(s, \pi).
 \end{align*}
\end{thm}
Applying this to the Artin representation above proves the existence of a cuspidal automorphic representation that demonstrates that the conjectured bound is sharp.\\

As a side note, we explain how the Ramanujan conjecture also implies a bound for the occurrence of another set of Hecke eigenvalues. Such a bound is implicit in~\cite{Ra97}. Our observation is that it is sharp, due to the same examples as above.

Given $\pi \in \mathcal{A}_0({\rm GL}_n(\A_F))$, define the set $S' = \{v \mid |a_v(\pi)| \neq n\}$. The Ramanujan conjecture for GL(n), implies the following inequality
\begin{align*}
\sum_{v} \frac{n ^2 - |a_v(\pi)|^2}{Nv^s} \leq \sum_{v \in S'}\frac{n ^2}{Nv^s}.
\end{align*}
Dividing by $\log \left(1/(s-1)\right)$ and taking the limit infimum as $s \rightarrow 1^+$, we obtain 
\begin{align} \label{eq1}
\frac{n ^2 -1}{n ^2} \leq \underline{\delta}(S').
\end{align}

This bound being implied by~\cite{Ra97}, let us recall the relevant theorem:
\begin{thm}[Ramakrishnan]
  Given a unitary cuspidal automorphic representation $\pi$ for GL(n),
then
\begin{align*}
  \underline{\delta}(\{v \mid |a_v(\pi)|\leq n\})\geq \frac{n^2 -1}{n^2}.
\end{align*}
\end{thm}

Ramakrishnan explained that the proof leads to a more general statement. Given a unitary cuspidal automorphic representation $\pi$ for GL(n) and a positive number $k$, we have
\begin{align*}
  \underline{\delta}(\{v \mid |a_v(\pi)|< k\})\geq \frac{k^2 -1}{k^2}.
\end{align*}
The inequality in the description of the set on the left-hand side was not originally strict, but the proof stands under this slight alteration. We point out that the bound here implies equation~(\ref{eq1}) above.\\

We show that this bound is sharp when $k=n$.
The projective image of the representation $\tau_n$ is an abelian group of order $n^2$, and thus the number of elements of $\tau_n (H)$ with trace zero is 
\begin{align*}
 \left(\frac{n^2-1}{n^2}\right) |\tau_n (H)|.
\end{align*}
The set of the remaining elements in $\tau_n(H)$ is the fibre (with respect to the projection ${\rm GL}_n(\C) \rightarrow {\rm PGL}_n(\C)$) over the identity element in the abelian group. Their images in ${\rm GL}_n(\C)$ are all of the form $e^{i \theta }I_n$, so a proportion of exactly $1/ n^2$ of the elements of $\tau_n (H)$ have a trace that has absolute value equal to $n$. The corresponding cuspidal automorphic representation $\pi=\pi(\tau_n)$ (via the theorem of Arthur--Clozel) then proves that equation~(\ref{eq1}) is sharp.

Because the bound is sharp, increasing it would only be possible for a proper subset of cuspidal automorphic representations. This is illustrated in~\cite{Ra97} for $n = 2$. There, the bound of 1/4 is sharp due to certain dihedral automorphic representations. By assuming the cuspidality of the symmetric square, one excludes the dihedral case and obtains a bound of 9/10 that holds for all non-dihedral unitary cuspidal automorphic representations for GL(2). Similarly when $n > 2$, increasing the bound will only be possible for a subset of unitary cuspidal automorphic representations for GL(n).

\section{Proof of Theorem~\ref{thmb}}\label{sthmb}\label{s6}

We consolidate the results of the previous sections to prove Theorem~\ref{thmb}.

\begin{thm2}
Let $\pi \in \mathcal{A}_0({\rm GL}_n(\A_F))$, where $n = 2$, or $\pi$ is essentially self-dual and $n = 3$.
Then the upper density of places $v$ at which the Hecke eigenvalue is zero is bounded above by $1-1/n^2$, where $n = 2$ or $3$, respectively.
\end{thm2}

\begin{proof}
As shown in Section~\ref{s3}, Theorem~\ref{thmc} implies Theorem~\ref{thmb} in all cases except when $\pi$ is an essentially self-dual cuspidal automorphic representation for GL(3) that is associated to a tetrahedral or octahedral automorphic representation (recall that `associated' was defined in Subsection~\ref{gl3}).

Since the Ramanujan conjecture holds for tetrahedral and octahedral automorphic representations, it therefore also holds for the essentially self-dual cuspidal automorphic representations for GL(3) associated to them. The proof in Section~\ref{sce}, which shows that the Ramanujan conjecture for $\pi$ implies a positive answer to Question~\ref{conja} for $\pi$, applies here. This concludes the proof of Theorem~\ref{thmb}.\\
\end{proof}

\subsection*{Acknowledgements}
The author would like to thank Dinakar Ramakrishnan for some productive discussions, Paul Nelson for his useful comments on an earlier draft of this paper, and the referees for their helpful comments on the strengthening of Theorem~\ref{thmd} as well as the exposition of the paper.

%\bibliography{mybib}{}

\begin{thebibliography}{00}

\bibitem{AC89}
Arthur, J. and L. Clozel.
\textit{Simple algebras, base change, and the advanced theory of the trace formula}, vol.~120 of {Annals of Mathematics Studies}.
Princeton University Press, Princeton, NJ, 1989.

\bibitem{BB11}
Blomer, V., and F. Brumley.
``On the {R}amanujan conjecture over number fields.''
\textit{Annals of Mathematics} 174, no. 1 (2011), 581--605.

\bibitem{CKPSS04}
Cogdell, J.~W., H.~H. Kim, I.~I. Piatetski-Shapiro, and F. Shahidi.
``Functoriality for the classical groups.''
\textit{Publications Math\'ematiques. Institut de Hautes \'Etudes
              Scientifiques}, 99 (2004), 163--233.

\bibitem{GJ78}
Gelbart, S., and H. Jacquet.
``A relation between automorphic representations of {${\rm GL}(2)$} and
  {${\rm GL}(3)$}.''
\textit{Annales Scientifiques de l'\'Ecole Normale Sup\'erieure} 11, no. 4 (1978), 471--542.

\bibitem{GRS99}
Ginzburg, D., S. Rallis and D. Soudry.
``On explicit lifts of cusp forms from {${\rm GL}_m$} to classical
  groups.''
\textit{Annals of Mathematics} 150, no. 3 (1999), 807--866.

\bibitem{JS76}
Jacquet, H., and J.~A. Shalika.
``A non-vanishing theorem for zeta functions of {${\rm GL}_{n}$}.''
\textit{Inventiones Mathematicae} 38, no. 1 (1976/77), 1--16.

\bibitem{JS81}
Jacquet, H., and J.~A. Shalika.
``On {E}uler products and the classification of automorphic forms. {II}''
\textit{American Journal of Mathematics} 103, no. 4 (1981), 777--815. 

\bibitem{Ki03}
Kim, H.~H.
``Functoriality for the exterior square of {${\rm GL}_4$} and the
  symmetric fourth of {${\rm GL}_2$}.''
\textit{Journal of the American Mathematical Society} 16, no. 1 (2003), 139--183 (electronic).
With appendix 1 by Ramakrishnan, D. and appendix 2 by Kim, H.~H. and P. Sarnak.

\bibitem{KS00}
Kim, H.~H., and F. Shahidi.
``Functorial products for {$\rm GL_2\times GL_3$} and functorial
  symmetric cube for {$\rm GL_2$}.''
\textit{Comptes Rendus de l'Acad\'emie des Sciences. S\'erie I.
              Math\'ematique} 331, no. 8 (2000), 599--604.

\bibitem{KS02}
Kim, H.~H., and F. Shahidi.
``Cuspidality of symmetric powers with applications.''
\textit{Duke Math. Journal} 112, no. 1 (2002), 177--197.

\bibitem{La80}
Langlands, R.~P.
\textit{Base change for {${\rm GL}(2)$}}, vol.~96 of \textit{Annals of Mathematics Studies}.
Princeton University Press, Princeton, N.J., 1980.

\bibitem{LRS99}
Luo, W., Z. Rudnick and P. Sarnak.
``On the generalized {R}amanujan conjecture for {${\rm GL}(n)$}.''
In \textit{Automorphic forms, automorphic representations, and
  arithmetic ({F}ort {W}orth, {TX}, 1996)}, vol.~66 of \textit{Proc. Sympos. Pure
  Math.} Amer. Math. Soc., Providence, RI, 1999, pp.~301--310.

\bibitem{PR12}
Prasad, D., and D. Ramakrishnan.
``On the cuspidality criterion for the Asai transfer to GL(4).''
{A}ppendix to: ``Determination of cusp forms on {$GL(2)$} by coefficients restricted to quadratic subfields.''
In \textit{Journal of Number Theory} 132, no. 6 (2012), 1359--1384.

\bibitem{Ra94}
Ramakrishnan, D.
``A refinement of the strong multiplicity one theorem for {${\rm
  GL}(2)$}. {A}ppendix to: ``{$l$}-adic representations associated to modular
  forms over imaginary quadratic fields. {II}'' [Inventiones Mathematicae\ 116
  (1994), no.\ 1-3, 619--643; {MR}1253207 (95h:11050a)] by {R}. {T}aylor.
\textit{Inventiones Mathematicae} 116, no. 1-3 (1994), 645--649.

\bibitem{Ra97}
Ramakrishnan, D.
``On the coefficients of cusp forms.''
\textit{Mathematical Research Letters} 4, no. 2-3 (1997), 295--307.

\bibitem{Ra00}
Ramakrishnan, D.
``Modularity of the {R}ankin-{S}elberg {$L$}-series, and multiplicity
  one for {${\rm SL}(2)$}.''
\textit{Annals of Mathematics} 152, no. 1 (2000), 45--111.

\bibitem{RW04}
Ramakrishnan, D., and S. Wang
``A cuspidality criterion for the functorial product on {$\rm GL(2)\times GL(3)$} with a cohomological application.''
\textit{International Mathematics Research Notices} 27, no. 27 (2004), 1355--1394.

\bibitem{Ra09}
Ramakrishnan, D.
``An exercise concerning the selfdual cusp forms on {${\rm GL}(3)$}.''
preprint.

\bibitem{Se81}
Serre, J.-P.
``Quelques applications du th\'eor\`eme de densit\'e de {C}hebotarev.''
\textit{Institut des Hautes \'Etudes Scientifiques. Publications
              Math\'ematiques}, 54 (1981), 323--401.

\bibitem{Tu81}
Tunnell, J.
``Artin's conjecture for representations of octahedral type.''
\textit{American Mathematical Society. Bulletin. New Series} 5, no. 2 (1981), 173--175.

\bibitem{NW2}
Walji, N.
``Further refinement of strong multiplicity one for {${\rm GL}(2)$}.''
\textit{Transactions of the American Mathematical Society}, 366 (2014), 4987--5007.

\end{thebibliography}
%\bibliographystyle{amsalpha}

\end{document}